\documentclass[11pt]{amsart}

\usepackage[T1]{fontenc}
\usepackage[utf8]{inputenc}
\usepackage{lmodern}
\usepackage{amsmath,amssymb,mathtools}
\usepackage{microtype}
\usepackage[colorlinks=true,
            linkcolor=blue,
            citecolor=blue,
            urlcolor=blue]{hyperref}

\newtheorem{theorem}{Theorem}[section]
\newtheorem{proposition}[theorem]{Proposition}
\newtheorem{lemma}[theorem]{Lemma}
\newtheorem{corollary}[theorem]{Corollary}
\theoremstyle{definition}
\newtheorem{definition}[theorem]{Definition}
\theoremstyle{remark}
\newtheorem{remark}[theorem]{Remark}

\newcommand{\N}{\mathbb N}
\newcommand{\Id}{\operatorname{Id}}

\newcommand{\oplusinf}{\mathbin{\oplus_{\infty}}}
\newcommand{\rnorm}[1]{\lVert #1\rVert_{\mathrm r}}

\title[A positive version of Miljutin's theorem]
      {A positive version of Miljutin's theorem}

\date{\today}

\thanks{The first two authors were supported by Fundaci\'{o}n S\'{e}neca - ACyT Regi\'{o}n de Murcia project 21955/PI/22 and by Agencia Estatal de Investigación (Government of Spain) and ERDF project PID2021-122126NB-C32. Laguna-Ricart was supported by the University of Murcia through its Predoctoral Contracts Programme under the Research Promotion Plan (“Programa de Contratos Predoctorales del Plan de Fomento de la Investigación de la Universidad de Murcia”). Research of the fourth author has been partially funded by projects PID2023-146505NB-C21 and PID2024-162214NB-100 funded by MICIU/AEI/ 10.13039/501100011033.}


\author[Laguna-Ricart]{Javier Laguna-Ricart}
\address{Universidad de Murcia, Departamento de Matemáticas\\
Campus de Espinardo \\
30100 Murcia, Spain.}
\email{javier.l.r@um.es}

\author[Martínez-Cervantes]{Gonzalo Martínez-Cervantes}
\address{Universidad de Murcia, Departamento de Matemáticas\\
Campus de Espinardo,
30100 Murcia, Spain.}
\email{gonzalo.martinez2@um.es}

\author[Rondo\v s]{\\ Jakub Rondo\v s}
\address{Department of Mathematics, Faculty of Electrical Engineering, Czech Technical
University in Prague, Technicka 2, 166 27 Prague 6, Czechia.}
\email{jakub.rondos@gmail.com}

\author[Salguero-Alarc\'on]{Alberto Salguero-Alarc\'on}
\address{Departamento de An\'alisis Matem\'atico y Matem\'atica Aplicada, 
Universidad Complutense de Madrid, 
Plaza de las Ciencias 3, 
28040 Madrid, Spain.}
\email{albsalgu@ucm.es}

\begin{document}

\begin{abstract}
We provide a positive version of Miljutin's theorem. Namely, we
show that for any uncountable compact metric spaces $K$ and $L$ there exists a
positive isomorphism $T:C(K)\longrightarrow C(L)$.
Moreover, the inverse of $T$ is a regular operator, and $T$ may be chosen so
that $\rnorm{T} \cdot \rnorm{T^{-1}}\leq 9+6\sqrt3$.
Our method also gives positive versions of classical nonmetrizable results for
Miljutin and Dugundji compacta. In particular, if $G$ and $H$ are infinite
compact groups of the same weight, then $C(G)$ and $C(H)$ are positively isomorphic with the same quantitative estimate.
\end{abstract}

\maketitle

\section{Introduction}

% All Banach lattices in this paper are real. For a compact Hausdorff space
% $K$, the symbol $C(K)$ denotes the Banach lattice of real-valued continuous
% functions on $K$, with the supremum norm and the pointwise order.

% We use the following one-sided
% terminology: a \emph{positive isomorphism} between Banach lattices is a
% bounded linear bijection $T$ such that $T\geq 0$. Its inverse is therefore not required to be positive. Indeed, if both $T$ and $T^{-1}$ were positive,
% then $T$ would be an order isomorphism. Kaplansky's theorem would then imply
% that the underlying compact spaces are homeomorphic
% \cite{Kaplansky1947}.

Miljutin's classical theorem states that whenever $K$ and $L$ are uncountable compact metrizable spaces, their spaces of continuous functions $C(K)$ and $C(L)$ are linearly
isomorphic
\cite{Miljutin1966,Pelczynski1968}.
The positive analogue of Miljutin's theorem is concerned instead with the existence of a \emph{positive  isomorphism} $T: C(K) \to C(L)$, that is, a linear isomorphism such that $Tf\geq 0$ whenever $f\geq 0$. This setting has been considered explicitly by C\'uth,
Havelka, Sar{\i} and the third author
\cite[Questions~0.1 and~5.2]{CuthHavelkaRondosSari2026}. In particular,
Question~5.2 asks about the existence of positive isomorphisms in both directions between
$C[0,1]$ and $C(2^{\N})$. 

\par The usual proof of Miljutin's theorem first obtains regular extensions and averaging operators and then invokes
Pe{\l}czy{\'n}ski's decomposition method. While the first step may suit our purposes concerning positive isomorphisms, the latter step does not: if $P$ is a positive projection, the complementary
projection $\Id-P$ is usually not. Our argument bypasses this difficulty by considering a new pair of elementary
positive isomorphisms which are \emph{not} inverses of one another. Additionally, their inverses are \emph{regular} operators, that is to say, they can be written as the difference of two positive operators. This is actually equivalent to be dominated by a positive operator, or that the regular norm
\[
\rnorm{S}
 =\inf\bigl\{\lVert R\rVert:R\geq0,\ -R\leq S\leq R\bigr\}
\]
is finite. 

% Recall that an operator between Banach lattices is \emph{regular} if it is the
% difference of two positive operators. We use the regular norm
% \[
% \rnorm{S}
%  =\inf\bigl\{\lVert R\rVert:R\geq0,\ -R\leq S\leq R\bigr\};
% \]
% see, for example, \cite{AliprantisBurkinshaw2006}. Thus
% $\rnorm{S}=\lVert S\rVert$ when $S\geq0$. 

Our main result is the following:

\begin{theorem}\label{thm:main-intro}
Let $K$ and $L$ be uncountable compact metrizable spaces. Then there is a
positive isomorphism 
$T:C(K)\longrightarrow C(L)$ such that $T^{-1}$ is regular. Moreover, $T$ can be chosen so that
\[
\|T\|\leq \sqrt{9+6\sqrt3}
\qquad\text{and}\qquad
\rnorm{T^{-1}}\leq \sqrt{9+6\sqrt3}.
\]
\end{theorem}

\par In combination with the results from \cite{CuthHavelkaRondosSari2026}, Theorem \ref{thm:main-intro} yields a complete classification of separable $\mathcal{C}(K)$ spaces by positive isomorphisms. The constants therein are obtained by keeping free parameters in the pair of positive isomorphisms and then perform a minimizing argument. It is worth mentioning that the same constant was obtained by Kania and Lewicki in
a recent quantitative decomposition theorem for any pair of Banach spaces which are
isometrically isomorphic to their squares and isometrically isomorphic to $1$-complemented subspaces of
one another \cite{KaniaLewicki2026}. In particular, their result improved the preceding bound
$(3+\sqrt2)^2$ of Korpalski and Plebanek for the distance between
$L_\infty[0,1]$ and $\ell_\infty$ \cite{KorpalskiPlebanek2025}. While our hypotheses
and construction are different, the optimization
produces the same constant.

\subsection{Content}  Section \ref{sec:retraction} contains the basic notion of \emph{positive retract} (or positively complemented subspace) and some easy consequences. Section \ref{sec:CB} contains the main ingredients for the proof of Theorem \ref{thm:main-intro}: a pair of positive isomorphisms with regular inverse and further properties concerning positivity, and the existence of positive isomorphisms that enable us to apply a positive analogue of Pe\l czy\'nski's decomposition method. The proof of Theorem \ref{thm:main-intro} appears in Section \ref{subsec:proof}.
\par Finally, in Section \ref{sec:non-metrizable} we observe that the main ingredient of the proof of our main result can survive outside the separable setting. Precisely, combining our argument with
the work of Haydon on non-metrizable compacta \cite{Haydon1974,Haydon1976} and
classical results on dyadic compacta and compact groups yields further positive
versions of Pe{\l}czy{\'n}ski's isomorphic classification results. As a consequence we prove, for instance, that if $K$ and $L$ are infinite compact groups of equal weight, then there exists a positive isomorphism $T: C(K) \to C(L)$ with regular inverse.

\subsection{Preliminaries} 
All Banach lattices in this paper are real. We remark that the inverse of a positive isomorphism $T: X \to Y$ between Banach lattices is not required to be positive; in fact, if both $T$ and $T^{-1}$ were positive,
then $T$ would be a lattice isomorphism. This would imply, in the case $X=C(K)$ and $Y=C(L)$, that the underlying compact spaces $K$ and $L$ are homeomorphic by Kaplansky's theorem
\cite{Kaplansky1947}. Also, recall that every norm-one operator $T: C(K) \to C(L)$  such that $T1 = 1$ is automatically positive, since for every non-negative $f$ in the unit ball of $C(K)$ we have $\|T(1-f)\| = \|1 - Tf\| \leq 1$,
and therefore $Tf\geq 0$. 
\par On the other hand, we call a linear isomorphism $T$ between Banach lattices \emph{regular} if both $T$
and $T^{-1}$ are regular operators, and call two Banach lattices
\emph{regularly isomorphic} if such an isomorphism exists. 
We depart from the ``one-sided terminology'' of positive operators in order to be consistent with the literature; see, for example,
\cite{Flores2002}.

\section{Positive retractions} \label{sec:retraction}

\begin{definition}\label{def:positive-retract}
Let $X$ and $Y$ be Banach lattices. We call $X$ a \emph{positive retract} of
$Y$ if there are positive operators $J:X\longrightarrow Y$ and $Q:Y\longrightarrow X$
such that $QJ=\Id_X$. If $J$ and $Q$ are contractions, we call $X$ a
\emph{positive contractive retract} of $Y$.
\end{definition}

\begin{remark} \label{rem:equivalence_positive_retract}A reformulation of Definition \ref{def:positive-retract} in terms of complemented subspaces is possible. Indeed, $X$ is a positive retract of $Y$ if and only if it is lattice isomorphic to the range of a positive projection $P:Y\to Y$. Recall that, in general, $P(Y)$ is not necessarily a sublattice of $Y$, but it is a Banach lattice under the order induced by $Y$ and the norm $\|z\|_\sim = \|P(|z|)\|_Y$, which is clearly equivalent to the norm induced by $Y$ (see \cite[III, Proposition 11.5]{schaefer} and \cite[Theorem 3.1]{AbramovichAliprantisPolyrakis1994}).
\par Now, if $J: X \to Y$ and $Q: Y \to X$ are positive operators satisfying $QJ=\Id_X$, then $J(X)$ is the range of the positive projection $P=JQ$, and $J^{-1}: J(X) \to X$ is also positive, since if $J(x) \geq 0$, then $x = Q[J(x)] \geq 0$. Hence $J$ is a lattice isomorphism. Conversely, given a positive projection $P: Y \to Y$ and a lattice isomorphism $J: X \to P(Y)$, the operator $Q=J^{-1}P$ witness the fact that $X$ is a positive retract of $Y$.  
\par Observe that, in the contractive case, this argument yields that $X$ is a positive contractive retract of $Y$ if and only if it embeds positively and isometrically onto a positively $1$-complemented subspace of $Y$. 
\end{remark}

% \begin{remark}
% \label{rem:equivalence_positive_retract}
% The contractive version of Definition~\ref{def:positive-retract} is
% equivalent to a more standard formulation in terms of complemented subspaces \cite{AbramovichAliprantisPolyrakis1994,Tsekrekos1982}, as we now indicate. We will however retain the retraction terminology because we deem it more suitable for the proofs below. 
% \par If $J$ and $Q$ are positive
% contractions and $QJ=\Id_X$, then
% \[
%  \lVert x\rVert=\lVert QJx\rVert\leq\lVert Jx\rVert\leq\lVert x\rVert,
% \]
% so $J$ is a positive isometry and $P=JQ$ is a positive contractive
% projection onto $J(X)$. Conversely, suppose that $J:X\to Y$ is a positive
% isometric embedding and that $P:Y\to J(X)$ is a positive contractive
% projection. By \cite[Theorem~4]{Gutman2025}, $J^{-1}$ is positive; hence
% $Q=J^{-1}P$ is positive and contractive and $QJ=\Id_X$. Thus $X$ is a
% positive contractive retract of $Y$ if and only if it embeds positively and
% isometrically onto a positively $1$-complemented subspace of $Y$.
% The word \emph{subspace} is intentional: the range of a positive projection
% need not be a sublattice of the ambient lattice.
% \end{remark}

\begin{proposition}\label{prop:positive-retract}
Let $K$ be an uncountable compact metrizable space and let $L$ be a nonempty
compact metrizable space. Then $C(L)$ is a positive contractive retract of
$C(K)$.
\end{proposition}

\begin{proof}
Rosenthal had already observed in \cite[Theorem~2.8]{Rosenthal2003} that most ingredients from the proof of Miljutin's theorem survive in the positive setting. Indeed, following \cite[\S 4.4]{AK16}, we have:
\begin{itemize}
    \item Borsuk's theorem \cite[Theorem 4.4.4]{AK16} actually produces positive operators. Since $L$ embeds as a closed subspace of $[0,1]^\mathbb N$ and $K$ contains a closed copy of the Cantor set $\Delta$, $C(L)$ is  a positive contractive retract of $C([0,1]^\mathbb N)$, and $C(\Delta)$ is also a positive contractive retract of $C(K)$. 
    \item Miljutin's lemma \cite[Lemma 4.4.7]{AK16} features a surjection $\psi: \Delta \to [0,1]$ and a norm-one operator $R:C(\Delta)\to C[0,1]$ such that $R(f\circ\psi)=f$ for all $f\in C[0,1]$. Such $R$ is actually denoted as $VT$ in the proof of \cite[Lemma 4.4.7]{AK16}, and it satisfies $R1 =1$, so it is positive.    
    \item The same situation happens in \cite[Theorem 4.4.8]{AK16}: there is a surjection $\tilde\psi:\Delta\to[0,1]^\mathbb N$ and a norm-one operator $\tilde R:C(\Delta)\to C([0,1]^\mathbb N)$ such that $\tilde R1=1$ and $\tilde R(f\circ\tilde\psi)=f$ for all $f\in C([0,1]^\mathbb N)$. Hence $C([0,1]^\mathbb N)$ is a positive contractive retract of $C(\Delta)$.
\end{itemize}
As a consequence of the above three items, $C(L)$ is a positive contractive retract of $C(K)$. 
\end{proof}

% \begin{remark}
% Rosenthal's formulation packages the two classical topological ingredients:
% regular extension from a Cantor subset and a regular averaging map from the
% Cantor set onto $L$. For the underlying results see
% \cite{Dugundji1951,Miljutin1966,Pelczynski1968,Ditor1970}.
% \end{remark}

For our next result, we need to consider $\alpha\N=\N\cup\{\infty\}$,  the one-point compactification of the discrete space $\N$. 
Given a Banach lattice $E$, let
\[
c(E)=\bigl\{(x_n)_{n=1}^\infty:x_n\text{ converges in norm in }E\bigr\},
\]
with the supremum norm and the coordinatewise order. It is clear that $c(E)$ is a Banach lattice. Moreover, the following fact is well-known:

\begin{lemma}\label{lem:cC}
For every compact Hausdorff space $L$, $
c\bigl(C(L)\bigr)$ and $C(\alpha\N\times L)$
are isometrically lattice isomorphic.
\end{lemma}

\begin{proof}
For $F\in C(\alpha\N\times L)$,  set $\Theta F=\bigl(F(n,\cdot)\bigr)_{n=1}^\infty$.
We claim that $F(n,\cdot)$ converges uniformly to $F(\infty,\cdot)$ in $C(L)$. Indeed, put
\[
G(m,t)=F(m,t)-F(\infty,t)
\qquad (m\in\alpha\N,\ t\in L).
\]
The function $G$ is continuous and vanishes on the compact slice
$\{\infty\}\times L$. Given $\varepsilon>0$, finitely many product
neighborhoods on which $|G|<\varepsilon$ cover that slice. Intersecting their
$\alpha\N$-coordinates gives a neighborhood of $\infty$ on which
$|G(m,t)|<\varepsilon$ for every $t\in L$. Hence
$F(n,\cdot)\to F(\infty,\cdot)$ uniformly and
$\Theta F\in c(C(L))$.

Conversely, if $(f_n)_{n=1}^\infty$ converges uniformly in $C(L)$ to $f_\infty$, define
$F(n,t)=f_n(t)$ and $F(\infty,t)=f_\infty(t)$. Uniform convergence gives
continuity along $\{\infty\}\times L$, while continuity elsewhere is
immediate. These two assignments are inverse lattice isometries.
\end{proof}

\begin{corollary}\label{cor:sequence-retracts}
Let $K$ and $L$ be uncountable compact metrizable spaces. Then
$c(C(L))$ is a positive contractive retract of $C(K)$.
\end{corollary}

\begin{proof}
The space $\alpha\N\times L$ is an uncountable metrizable compactum. Hence we conclude by virtue of 
Proposition~\ref{prop:positive-retract} and Lemma~\ref{lem:cC}.
\end{proof}

\section{A positive Cantor-Bernstein principle} \label{sec:CB}

We begin our preparation for the proof of Theorem \ref{thm:main-intro}, which will come out as a consequence of an analogue of the Cantor-Bernstein principle for positive operators. 

\subsection{Two elementary isomorphisms}
The first proposition is the elementary core of the argument. For any Banach
lattice $E$, it is \emph{very} easy to see that the spaces $c(E)$ and $c(E)\oplusinf E$ are lattice isometric.
However, the canonical lattice isometry will not work for our purposes. We therefore consider two different positive isomorphisms between $c(E)$ and $c(E)\oplusinf E$ with free
parameters that will later be optimized. Their inverses will not be positive in general, 
but are regular with explicit positive majorants as we now show. 

\begin{proposition}\label{prop:sequence-isomorphisms}
Let $E$ be a Banach lattice. Consider $a,b>1$ and $c>0$. 

\begin{enumerate}
\item[\emph{(i)}] The operator
\[
 \Phi_E^{a}:c(E)\longrightarrow c(E)\oplusinf E,
 \qquad
 \Phi_E^{a}z=(\phi_E^a z,\rho_E z),
\]
given by 
\[ (\phi_E^a z)_n=z_n+az_{n+1}, \qquad \rho_E z=z_1, \]
is a positive isomorphism satisfying $
 \phi_E^a-\Id_{c(E)}\geq0$.
Moreover,
\[
 \lVert\Phi_E^{a}\rVert= 1+a,
 \qquad
 \rnorm{(\Phi_E^{a})^{-1}}
 \leq\max\left\{1,\frac1{a-1}\right\}.
\]

\item[\emph{(ii)}] The operator
\[
 \Psi_E^{b,c}:c(E)\oplusinf E\longrightarrow c(E),
\]
given by
\begin{equation*}
 [\Psi_E^{b,c}(z,e)]_1=z_1+ce,
 \qquad
 [\Psi_E^{b,c}(z,e)]_n=z_n+bz_{n-1}\quad(n\geq2),
\end{equation*}
is a positive isomorphism. If
$\pi_1:c(E)\oplusinf E\to c(E)$ denotes the projection onto the first coordinate,
then $\Psi_E^{b,c}-\pi_1\geq0$.
Moreover,
\[
 \lVert\Psi_E^{b,c}\rVert= 1+\max\{b,c\},
 \qquad
 \rnorm{(\Psi_E^{b,c})^{-1}}
 \leq\max\left\{\frac1{b-1},\frac{b}{c(b-1)}\right\}.
\]
\end{enumerate}
\end{proposition}

\begin{proof}

It is clear that the operator $\Phi_E^a: c(E) \to c(E)\oplusinf E$ defined in \emph{(i)} is positive and has norm 
$1+a$. For its inverse, observe that, given $(y,e)\in c(E)\oplusinf E$, the equation
$\Phi_E^{a}z=(y,e)$ determines $z$ recursively. Explicitly,
\begin{equation}\label{eq:Phi-inverse}
 z_n=\frac{(-1)^{n-1}}{a^{n-1}}e
     +\sum_{m=1}^{n-1}\frac{(-1)^{m-1}}{a^m}y_{n-m}.
\end{equation}
The first term tends to zero, while the sum is a convolution of the convergent sequence $(y_n)_{n=1}^\infty$ with the summable sequence $((-1)^{m-1}a^{-m})_{m=1}^\infty$, whose sum is $1/(1+a)$. Hence, if $y_n\to y_\infty$, then $z_n\to y_\infty/(1+a)$, and so this formula
defines the inverse on $c(E)\oplusinf E$. Finally, observe that the operator 
$R_\Phi:c(E)\oplusinf E\to c(E)$ defined by
\begin{equation}\label{eq:Phi-majorant}
 [R_\Phi(y,e)]_n
 =\frac{1}{a^{n-1}}e
  +\sum_{m=1}^{n-1}\frac{1}{a^m}y_{n-m}.
\end{equation}
is well-defined: its coordinates converge to $y_\infty/(a-1)$. It is positive and majorizes $(\Phi_E^{a})^{-1}$, with
\[ \lVert R_\Phi\rVert
 \leq\max\left\{1,\frac1{a-1}\right\}.
\]

\par As for \emph{(ii)}, it is also clear that the operator $\Psi_E^{b,c}: c(E) \oplusinf E \to c(E)$ defined there is positive with norm $1+\max\{b,c\}$, and it is immediate that $\Psi_E^{b,c}-\pi_1\geq 0$. Now, to obtain its inverse, pick 
$y=(y_n)_{n=1}^\infty\in c(E)$ and put
\begin{equation}\label{eq:Psi-inverse}
 z_n=\sum_{j=1}^{\infty}\frac{(-1)^{j-1}}{b^j}y_{n+j},
 \qquad
 e=\frac1c(y_1-z_1).
\end{equation}
By an analogous argument as in part \emph{(i)}, we have that if $y_n \to y_\infty$ then $z_n \to y_\infty/(1+b)$, and $z_1+ce=y_1$ and $z_n+bz_{n-1}=y_n$ for $n\geq2$.
Consequently, formula \eqref{eq:Psi-inverse} defines the inverse of $\Psi_E^{b,c}$. Now, consider the positive operator
$R_\Psi:c(E)\to c(E)\oplusinf E$ given by $R_\Psi y=(v,f)$, where
\begin{equation}\label{eq:Psi-majorant}
 v_n=\sum_{j=1}^{\infty}\frac1{b^j}y_{n+j},
 \qquad
 f=\frac1c(y_1+v_1).
\end{equation}
Then it is clear that \[
 -R_\Psi\leq(\Psi_E^{b,c})^{-1}\leq R_\Psi, \qquad 
 \lVert R_\Psi\rVert
 \leq\max\left\{\frac1{b-1},\frac{b}{c(b-1)}\right\}. \qedhere
\]
\end{proof}

\begin{remark}\label{rem:one-sided}
The inverse formulas \eqref{eq:Phi-inverse} and \eqref{eq:Psi-inverse}
contain alternating signs. Thus the construction is intrinsically one-sided:
the operators are positive, but their inverses need not be.
\end{remark}

\subsection{An absorption result for positive retracts} 
We now employ the construction of the previous section to produce a positive analogue of Pe\l czy\'nski decomposition method.

\begin{theorem}[Positive absorption]\label{thm:absorption}
Let $Z$ and $E$ be Banach lattices. Suppose that $c(E)$ is a positive retract
of $Z$, and denote
\[
J:c(E)\longrightarrow Z,
\qquad
Q:Z\longrightarrow c(E)
\]
the corresponding positive bounded operators with $QJ=\Id_{c(E)}$.
For every $a,b>1$ and $c>0$ there are positive isomorphisms
\[
\mathcal A_{a}:Z\longrightarrow Z\oplusinf E,
\qquad
\mathcal B_{b,c}:Z\oplusinf E\longrightarrow Z,
\]
 whose inverses are regular. In fact, if $J$ and $Q$ are contractions, then 
\[
\begin{aligned}
 \lVert\mathcal A_{a}\rVert
 &\leq 1+a,&
 \rnorm{\mathcal A_{a}^{-1}}
 &\leq2+\max\left\{1,\frac1{a-1}\right\},\\
 \lVert\mathcal B_{b,c}\rVert
 &\leq1+\max\{b,c\},&
 \rnorm{\mathcal B_{b,c}^{-1}}
 &\leq\max\left\{2+\frac1{b-1},
                   \frac{b}{c(b-1)}\right\}.
\end{aligned}
\]
\end{theorem}

\begin{proof} First of all, recall that the hypothesis yields an isomorphism 
\[ U: \ker Q \oplusinf c(E) \to Z, \quad U(w,z) = w+Jz.\]

To define $\mathcal A_{a}:Z\to Z\oplusinf E$, we use the isomorphism $\Phi_E^{a}=(\phi_E^a,\rho_E)$ from
Proposition~\ref{prop:sequence-isomorphisms}. 
Consider the map
\begin{equation}\label{eq:A}
\mathcal A_{a} z
 =\left(z+J(\phi_E^a-\Id_{c(E)})Qz,\,\rho_E Qz\right), 
\end{equation}
which is positive, since $\phi_E^a - \Id_{c(E)} \geq 0$. Now, observe that 
\[ \mathcal A_{a} U(w,z) = (U(w,\phi_E^az), \rho_E z), \]
in other words, 
\[ \mathcal A_{a} = (U \oplus \Id_E) \circ(\Id_{\ker Q} \oplus \Phi_E^{a})\circ U^{-1}, \]
where the involved operators act coordinatewise as indicated by the parentheses:
\begin{align*}
   U\oplusinf  \Id_E: &\big(\!\ker Q \oplusinf  c(E)\big)\oplusinf  E \to Z \oplusinf  E, \\ \Id_{\ker Q} \oplusinf  \Phi_E^{a}: &\ker Q \oplusinf c(E) \to \ker Q \oplusinf  \big(c(E) \oplusinf  E\big).
\end{align*}
Therefore, $\mathcal A_{a}$ is a positive isomorphism.
Similarly, define $\mathcal B_{b,c}: Z \oplusinf E \to Z$ by
\begin{equation}\label{eq:B}
\mathcal B_{b,c}(z,e)
 =z+J\bigl(\Psi_E^{b,c}(Qz,e)-Qz\bigr),\
\quad z\in Z,\ e\in E, 
\end{equation}
which is positive --again by Proposition \ref{prop:sequence-isomorphisms}-- and satisfies
\[
\mathcal B_{b,c}(U(x,u),e)
 =U\bigl(x,\Psi_E^{b,c}(u,e)\bigr).
\]
In terms of operators, this means that 
\[ \mathcal B_{b,c} = U \circ(\Id_{\ker Q} \oplusinf \Psi_E^{b,c})\circ (U \oplusinf \Id_E)^{-1}, \]
and so $\mathcal B_{b,c}$ is an isomorphism. 

\par We now check regularity of $\mathcal A_{a}^{-1}$ and $\mathcal B_{b,c}^{-1}$ and obtain the corresponding estimations for their regular norms. For $(w,e)\in Z\oplusinf E$, put $u=(\Phi_E^{a})^{-1}(Qw,e)$. Then
\begin{equation*}
%\label{eq:A-inverse}
\mathcal A_{a}^{-1}(w,e)=w-JQw+Ju.
\end{equation*}
The operator $\Id_Z-JQ$ is regular and is dominated by the positive operator
$\Id_Z+JQ$. Combining this with the majorant $R_\Phi$ from
\eqref{eq:Phi-majorant}, we see that $\mathcal A_{a}^{-1}$ is dominated by
the positive operator
\[
(w,e)\longmapsto
w+JQw+J R_\Phi(Qw,e).
\]
Hence $\mathcal A_{a}^{-1}$ is regular. Likewise, for $w\in Z$, write $
(\Psi_E^{b,c})^{-1}(Qw)=(u,e)$.
Then
\begin{equation*}
%\label{eq:B-inverse}
\mathcal B_{b,c}^{-1}w=(w-JQw+Ju,e).
\end{equation*}
Using $R_\Psi$ from \eqref{eq:Psi-majorant}, the right-hand side is dominated
by the positive operator
\[
w\longmapsto\bigl(w+JQw+Jv,f\bigr),
\qquad \text{ where }
(v,f)=R_\Psi Qw.
\]
Hence $\mathcal B_{b,c}^{-1}$ is regular. Finally, assume now that $J$ and $Q$ are contractions. Since
$\lVert\phi_E^a-\Id_{c(E)}\rVert= a$ and
$\lVert\rho_E\rVert= 1$, \eqref{eq:A} gives
$
\lVert\mathcal A_{a}\rVert\leq1+a
$.
The positive majorant for $\mathcal A_{a}^{-1}$ displayed above has norm
at most
\[
2+\max\left\{1,\frac1{a-1}\right\},
\]
and hence gives the asserted estimate for
$\lVert\mathcal A_{a}^{-1}\rVert_r$. Analogously, \eqref{eq:B} gives
\[
\lVert\mathcal B_{b,c}\rVert\leq1+\max\{b,c\},
\]
while the positive majorant for $\mathcal B_{b,c}^{-1}$ displayed above has norm
at most
\[
\max\left\{2+\frac1{b-1},\frac{b}{c(b-1)}\right\},
\]
thus giving the asserted estimate for
$\lVert\mathcal B_{b,c}^{-1}\rVert_r$.
\end{proof}

% \begin{remark}\label{rem:complement}
% Recall that the projection $JQ$ onto $J(c(E))$ is positive, but the complementary
% projection $\Id_Z-JQ$ need not be. It appears in \eqref{eq:A-inverse} and
% \eqref{eq:B-inverse}, but only in the verification of invertibility. Positivity
% of the forward maps \eqref{eq:A} and \eqref{eq:B} follows instead from the
% positive perturbation properties \eqref{eq:Phi-positive} and
% \eqref{eq:Psi-positive}. This is precisely what avoids the obstruction in
% the ordinary decomposition argument.
% \end{remark}

\begin{theorem}[Positive Cantor--Bernstein principle]\label{thm:CB}
Let $X$ and $Y$ be Banach lattices. Suppose that $c(Y)$ is a positive retract
of $X$ and $c(X)$ is a positive retract of $Y$. Then there is a positive
isomorphism from $X$ onto $Y$ whose inverse is regular. 
\par Furthermore, if both retractions
are contractive, the isomorphism $T:X\to Y$ may be chosen so that 
$ \rnorm{T}\,\rnorm{T^{-1}}\leq9+6\sqrt3.$
\end{theorem}

\begin{proof}
Fix $a,b>1$ and $c>0$. Apply the first isomorphism in
Theorem~\ref{thm:absorption} to the positive retract $c(Y)$ of $X$, obtaining
a positive isomorphism
\[
\mathcal A_X^{a}:X\longrightarrow X\oplusinf Y.
\]
Apply the second isomorphism in Theorem~\ref{thm:absorption} to the positive
retract $c(X)$ of $Y$, obtaining a positive isomorphism
\[
\mathcal B_Y^{b,c}:Y\oplusinf X\longrightarrow Y.
\]
The coordinate flip
\[
\Sigma:X\oplusinf Y\longrightarrow Y\oplusinf X,
\qquad
\Sigma(x,y)=(y,x),
\]
is even a lattice isomorphism. Therefore
\[
T_0=\mathcal B_Y^{b,c}\, \Sigma\, \mathcal A_X^{a}:X\longrightarrow Y
\]
is a positive isomorphism with regular inverse. This proves the first
assertion.

We now suppose that all retracts
\begin{align*} J_Y: c(X) \to Y, & \quad Q_Y:Y \to c(X), \\ J_X: c(Y) \to X, & \quad Q_X: X \to c(Y),
\end{align*}
are contractions, and compute the following estimates for the regular norms of $T_0$ and its inverse:
\begin{align}
    \label{eq:CB-forward}
     \lVert T_0\rVert & \leq
 C_+(a,b,c):=1+\max\{c(1+a),b\}. \\[2mm] 
 \label{eq:CB-inverse}
 \rnorm{T_0^{-1}} & \leq C_-(a,b,c)
 :=\frac{2b}{c(b-1)}
   +\max\left\{
       \frac{2b-1}{(b-1)},
       \frac{b}{c(b-1)(a-1)}
     \right\}.
\end{align}
The estimate \eqref{eq:CB-forward} is straightforward: if $x\in X$, then 
$\mathcal A_X^{a}x=(w,e)$ with
$\lVert w\rVert\leq(1+a)\lVert x\rVert$ and
$\lVert e\rVert\leq \lVert x\rVert$, and therefore $\|\mathcal B^{b,c}_Y(e,w)\|\leq \|x\|+\max\{c(1+a), b\}\|x\|$. 
Hence we turn our attention to \eqref{eq:CB-inverse}. We first work with $(\mathcal B^{b,c}_Y)^{-1}$, using the positive majorant $R_\Psi:c(X)\to c(X)\oplusinf X$ of $(\Psi_X^{b,c})^{-1}$ defined in \eqref{eq:Psi-majorant}. Let $y\in Y_+$ and write 
\begin{align*}
 (u,e)=(\Psi_X^{b,c})^{-1}Q_Yy,
 \qquad
 w=y-J_YQ_Yy+J_Yu, \\
 (v,f) = R_\Psi Q_Y y, \qquad  h=y+J_YQ_Yy+J_Yv. 
\end{align*}
The domination of $(\Psi_X^{b,c})^{-1}$ by $R_\Psi$ gives $|u|\leq v$ and $|e|\leq f$, hence $|w|\leq h$, and by the definition of $R_\Psi$ and the contractivity of the
retraction operators,
\[
 \|v\|\leq\frac{1}{b-1}\|y\|,
 \qquad
 \|f\|\leq\frac{b}{c(b-1)}\|y\|.
\]
Therefore, 
\[
 \|h\|
 \leq \|y\|+\|Q_Yy\|+\|v\|
 \leq\frac{2b-1}{b-1}\|y\|.
\]
Now we deal with $(\mathcal A^a_X)^{-1}$, using the positive majorant $R_\Phi:c(Y)\oplusinf Y\to c(Y)$ of $(\Phi_Y^a)^{-1}$ considered in \eqref{eq:Phi-majorant}. Write
\[
 z=(\Phi_Y^{a})^{-1}(Q_Xe,w), \qquad 
 r=R_\Phi(Q_Xf,h).
\]
Since $Q_X$ is positive and $R_\Phi$ dominates $(\Phi_Y^{a})^{-1}$, we have that $|z|\leq r$. Hence \eqref{eq:Phi-majorant} and the contractivity of $Q_X$
give
 \begin{align*}
 \|r\|
 &\leq
 \max\left\{
   \|h\|,
   \frac{\|Q_Xf\|}{a-1}
 \right\} \leq 
 \max\left\{
  \frac{2b-1}{b-1},
  \frac{b}{c(b-1)(a-1)}
 \right\}\|y\|.
 \end{align*}
To conclude, observe that, with the preceding notations,
\[
 T_0^{-1}y=e-J_XQ_Xe+J_Xz.
\]
The assignments $y\mapsto f$, $y\mapsto h$, and $y\mapsto r$ define positive operators. Therefore, so does the map $R:Y\longrightarrow X$ given by
\[  Ry=f+J_XQ_Xf+J_Xr. \]
Now, since $|e|\leq f$ and $
|z|\leq r$, we obtain that $|T_0^{-1}y|\leq Ry$ for every $y\in Y_+$. Therefore, $T_0^{-1}$ is dominated by $R$, and this finally provides estimation \eqref{eq:CB-inverse}: 
\begin{align*}
 \|Ry\|
 &\leq
 \|f\|+\|J_XQ_Xf\|+\|J_Xr\|              \leq
 2\|f\|+\|r\|                                        \\
 &\leq
 \left[
 \frac{2b}{c(b-1)}
 +\max\left\{
       \frac{2b-1}{b-1},
       \frac{b}{c(b-1)(a-1)}
      \right\}
 \right]\|y\|.
\end{align*}

\par To finish, we minimize $P(a,b,c)$ the product of 
\begin{align*} 
C_+(a,b,c) =& \big(1+\max\{c(1+a), b\}\big), \\[1mm]  C_-(a,b,c) = &\frac{2b}{c(b-1)}+\max\left\{\frac{2b-1}{b-1}, \frac{b}{c(b-1)(a-1)}\right\}.\end{align*}
We give an elementary and relatively short argument for the sake of completeness. Recall that $a,b>1$ and $c>0$. First, observe that given $(a,b,c)$ such that $b<c(1+a)$, letting $b^* = c(1+a)$, we have $P(a,b^*,c) < P(a,b,c)$: indeed, the value of $C_+$ remains unchanged, but $C_-$ is a decreasing function of $b$. Hence there is a minimum of $P$ with $b\geq c(1+a)$, and so we can assume that $C_+(a,b,c) = 1+b$. Now, as a function of $c\in(0,\frac{b}{1+a}]$, $C_-(a,b,c)$ is strictly decreasing. Hence there is a minimum of $P$ satisfying $b=c(1+a)$. In other words, minimizing $P$ amounts to minimizing 
\[ P_1(a,b) = \frac{1+b}{b-1}\left[ 2(1+a)+\max\left\{2b-1, \frac{1+a}{a-1}\right\}\right]. \]
Let $a^*$ be given by 
\[ 2b-1 = \frac{1+a^*}{a^*-1}\]
(that is, $a^* = \frac{b}{b-1}$). The minimum cannot be attained at a point $(a,b)$ with $a>a^*$, since $P_1$ is a strictly increasing function of $a$ in $[a^*, +\infty)$. Therefore, there is a minimum with $a\leq a^*$, and this allows us to minimize $P_1$ by minimizing
\begin{align*} P_2(a,b) & = \frac{1+b}{b-1}\left[ 2(1+a)+\frac{1+a}{a-1}\right] \quad \text{ subject to } \quad a\leq \frac{b}{b-1}. \end{align*} 
But $P_2(a,b)$ is the product of a decreasing function in $b \in (1, \frac{a}{a-1}]$ and a function in $a$, hence there is a minimum with $b= \frac{a}{a-1}$ (or equivalently, $a=\frac{b}{b-1}$). Consequently, writing $P_2$ as a function of $a$ in $P_2$ yields
\[ P_3(a) = \frac{(a+1)(2a-1)^2}{a-1}, \]
and minimizing it we arrive at
\begin{equation*}
%\label{eq:optimal-parameters}
 a=\frac{1+\sqrt3}{2},
 \qquad b=2+\sqrt3,
 \qquad c=1+\frac1{\sqrt3}, 
\end{equation*}
which provide the following values: 
\[
 C_+=3+\sqrt3,
 \qquad
 C_-=\frac{3+3\sqrt3}{2},
 \qquad
 P=C_+C_-=9+6\sqrt3. 
\]
Multiplying $T_0$ by the positive scalar
$\lambda=(C_-/C_+)^{1/2}$ gives $T=\lambda T_0$ and balances the two
estimates:
\[
 \rnorm{T},\ \rnorm{T^{-1}}
 \leq\sqrt{C_+C_-}=\sqrt{9+6\sqrt3}.\qedhere
\]
\end{proof}

\subsection{The positive Miljutin theorem} \label{subsec:proof}
With the above preparations, the proof of Theorem \ref{thm:main-intro} is simple. 
Put $X=C(K)$ and $
Y=C(L)$.
By Lemma~\ref{lem:cC}, $c(Y)$ and $C(\alpha\N\times L)$ are isometrically lattice isomorphic. Since $\alpha\N\times L$ is a metrizable compact, Proposition~\ref{prop:positive-retract}, applied with source
space $K$ and target space $\alpha\N\times L$, shows that $c(Y)$ is a positive
contractive retract of $X$. Similarly,
\[
c(X)\cong C(\alpha\N\times K),
\]
and Proposition~\ref{prop:positive-retract}, now applied with source space
$L$, shows that $c(X)$ is a positive contractive retract of $Y$.
Theorem~\ref{thm:CB} therefore gives a positive isomorphism
$T:X\to Y$ whose inverse is regular and which satisfies
\[
\rnorm{T},\ \rnorm{T^{-1}}
\leq\sqrt{9+6\sqrt3}. \qquad\qed
\]

\begin{remark}
In the notation for the one-sided positive Banach--Mazur quantity introduced
in \cite{CuthHavelkaRondosSari2026}, the proof also gives the uniform estimate
\[
d_{BM}^{+}\bigl(C(K),C(L)\bigr)\leq 9+6\sqrt3.
\]
\end{remark}

% \begin{remark}\label{rem:regular-isomorphism}
% The conclusion is stronger than a positive isomorphism in the one-sided
% sense: $T$ and $T^{-1}$ are both regular, thus $C(K)$ and $C(L)$ are regularly isomorphic in the terminology fixed in the introduction, with an isomorphism which is
% positive in the prescribed direction. The latter positivity is an additional
% asymmetric strengthening of regular isomorphism.
% \end{remark}

% \begin{remark}
% The theorem is genuinely one-sided. If both $T$ and $T^{-1}$ were positive,
% then $T$ would be an order isomorphism. Kaplansky's theorem would then imply
% that the underlying compact spaces are homeomorphic
% \cite{Kaplansky1947}. The alternating signs in the inverse formulas of
% Proposition~\ref{prop:sequence-isomorphisms} are exactly where this stronger
% requirement is avoided. Their positive majorants are, at the same time, what
% makes the inverse regular.
% \end{remark}

\section{Miljutin compacta, Dugundji compacta, and compact groups}
\label{sec:non-metrizable}

We finish with some nonmetrizable consequences of the preceding argument.
They are positive versions of results in Pe{\l}czy{\'n}ski's
classification programme \cite{Pelczynski1968}.
 If $\pi:K\to L$ is a continuous surjection between compact spaces, write
\[
\pi^*:C(L)\longrightarrow C(K),
\qquad
\pi^*f=f\circ\pi.
\]
A \emph{regular averaging operator} for $\pi$ is a positive operator
$A:C(K)\to C(L)$ satisfying $A\pi^*=\Id_{C(L)}$. Such an operator is unital
and has norm one. Similarly, a \emph{regular extension operator} for a closed
subspace $F$ of a compact space $K$ is a positive operator
$D:C(F)\to C(K)$ which is a right inverse to restriction and preserves the
constant function~$1$. This terminology agrees with the classical one: for
operators between real $C(K)$-spaces, every unital contraction is positive,
while every positive unital operator has norm one.

% The adjective “regular” in the classical expressions “regular averaging operator” and “regular extension operator” is part of the established terminology. Such operators are positive and hence regular in the Banach lattice sense, but the expressions additionally encode the relevant averaging or extension identities.

Let $\kappa$ be an infinite cardinal. A compact space $K$ of weight $\kappa$
is called a \emph{Miljutin compactum} if there is a continuous surjection $\pi:2^\kappa\longrightarrow K$
which admits a regular averaging operator. This is the terminology of
Pe{\l}czy{\'n}ski \cite{Pelczynski1968}. Additionally, $K$ is a \emph{Dugundji compactum} if for each embedding $K \hookrightarrow L$ there exists a regular extension operator $C(K) \to C(L)$. We shall use two classical facts:
products of Miljutin compacta are Miljutin, and every Dugundji compactum is a
Miljutin compactum \cite{Haydon1974}.

\begin{theorem}\label{thm:nonmetrizable}
Let $K$ and $L$ be Miljutin compacta of the same infinite weight $\kappa$.
Suppose that each of $K$ and $L$ contains a closed copy of $2^\kappa$. Then
there is a positive isomorphism $T:C(K)\longrightarrow C(L)$
with regular inverse and such that
$ \rnorm{T} \cdot \rnorm{T^{-1}}\leq 9+6\sqrt3$.
\end{theorem}

\begin{proof}
We verify the hypotheses of Theorem~\ref{thm:CB}. Since $L$ is Miljutin,
so is $\alpha\N\times L$. Indeed, one takes the product of a Miljutin map onto
$L$ with a metrizable Miljutin map onto $\alpha\N$ and uses
$2^\kappa\times2^{\N}\cong2^\kappa$; the corresponding product of the
probability kernels gives a regular averaging operator. Hence there are a
continuous surjection $\pi:2^\kappa\longrightarrow\alpha\N\times L$
and a regular averaging operator $A$ for $\pi$.

Choose a closed subset $F$ of $K$ homeomorphic to $2^\kappa$ and identify $F$ with $2^\kappa$ via a fixed homeomorphism. Under this identification, regard $\pi$ as a map on $F$ and $A$ as an operator on $C(F)$. The Cantor cube is a Dugundji compactum, so there is a regular extension operator
$D:C(F)\to C(K)$ for the inclusion $F\subseteq K$. If
$R:C(K)\to C(F)$ is the restriction map, then
\[
 J=D\pi^*:C(\alpha\N\times L)\longrightarrow C(K),
 \qquad
 Q=AR:C(K)\longrightarrow C(\alpha\N\times L)
\]
are positive contractions and $QJ=\Id$. By Lemma~\ref{lem:cC}, this says that
$c(C(L))$ is a positive contractive retract of $C(K)$. Interchanging $K$ and
$L$ gives the other retraction. Theorem~\ref{thm:CB} completes the proof.
\end{proof}

The property of containing a closed copy of $2^\kappa$ for some $\kappa$ has several useful classical sufficient
conditions. Efimov and, independently, Gerlits proved that a dyadic compactum
$K$ of weight $\kappa$ contains a copy of $2^\kappa$ if and only if it is not
the union of countably many closed subspaces of weight less than $\kappa$
\cite{Efimov1977,Gerlits1976}. In particular, a dyadic compactum of weight
$\kappa$ contains $2^\kappa$ whenever $\operatorname{cf}(\kappa)>\omega$.
Since Miljutin compacta are dyadic, we obtain the following.

\begin{corollary}\label{cor:dugundji}
Let $K$ and $L$ be Dugundji compacta of the same infinite weight $\kappa$.
If each contains a copy of $2^\kappa$ --in particular, if
$\operatorname{cf}(\kappa)>\omega$-- then the conclusion and the estimates of
Theorem~\ref{thm:nonmetrizable} hold.
\end{corollary}

\begin{proof}
Every Dugundji compactum is Miljutin by Haydon's theorem
\cite{Haydon1974}; apply Theorem~\ref{thm:nonmetrizable} and the preceding
closed-cube criterion.
\end{proof}

For compact groups no cofinality restriction is needed. The
Ivanovskii--Kuzminov theorem says that compact groups are dyadic
\cite{Kuzminov1959}, and compact groups are in fact Dugundji compacta
\cite{Uspenskii1990}. If $G$ is an
infinite compact group of weight $\kappa$, then every nonempty open subset of
$G$ has weight $\kappa$: finitely many translates of it cover $G$. Were $G$
a countable union of closed subspaces of smaller weight, the Baire category
theorem would force one of them to have nonempty interior, a contradiction.
The Efimov--Gerlits criterion therefore gives a closed copy of $2^\kappa$ in
$G$. We have proved the following positive group version of Pe{\l}czy{\'n}ski's
classification.

\begin{corollary}\label{cor:compact-groups}
Let $G$ and $H$ be infinite compact groups of the same weight. Then there is
a positive isomorphism
 $T:C(G)\longrightarrow C(H)$
with regular inverse and such that
\[\rnorm{T} \cdot \ \rnorm{T^{-1}}\leq 9+6\sqrt3.\]
In particular, $C(G)$ and $C(H)$ are regularly isomorphic.
\end{corollary}

We finally record a consequence for products, combining the
preceding results with the construction of Plebanek, Rondo\v{s}, and Sobota
\cite{PlebanekRondosSobota2026}.

\begin{corollary}\label{cor:products}
Let $K_1$ and $K_2$ be compact spaces.
\begin{enumerate}
\item[\emph{(i)}] If $K_1$ and $K_2$ are non-scattered, then, for every nonempty compact
metrizable space $L$, the space $C(L)$ is a positive contractive retract of
$C(K_1\times K_2)$.

\item[\emph{(ii)}] Suppose that both $K_1$ and $K_2$ map continuously onto an infinite
compact group $G$ of weight $\kappa$. Then, for every Miljutin compactum $L$
of weight $\kappa$, the space $C(L)$ is a positive contractive retract of
$C(K_1\times K_2)$.
\end{enumerate}
\end{corollary}

\begin{proof}
Theorem~1.1 of \cite{PlebanekRondosSobota2026} is stated in terms of an
isometric complemented copy of $C(G)$. Positivity was not the primary object
of interest there and is therefore not recorded in its statement. It does,
however, follow directly from the proof: the embedding is the composition
operator
\[
 J_G:C(G)\longrightarrow C(K_1\times K_2),
 \qquad J_Gh=h\circ\rho,
\]
and its left inverse has the form
\[
 Q_Gf(y)=\int_{K_1\times K_2}f\,d\phi(y),
\]
where $\phi(y)$ is a probability measure concentrated on $\rho^{-1}(y)$.
Thus $J_G$ and $Q_G$ are positive contractions and
$Q_GJ_G=\Id_{C(G)}$.

If $K_1$ and $K_2$ are non-scattered, they both map continuously onto the
circle group $\mathbb T$. Proposition~\ref{prop:positive-retract} shows that
$C(L)$ is a positive contractive retract of $C(\mathbb T)$. Composing the two
retractions proves~\emph{(i)}.

For~\emph{(ii)}, the argument preceding Corollary~\ref{cor:compact-groups} gives a
closed copy $F$ of $2^\kappa$ in $G$. Let
$\pi:2^\kappa\to L$ be a Miljutin map with regular averaging operator $A$.
Identifying $F$ with $2^\kappa$, take a regular extension operator
$D:C(F)\to C(G)$ and let $R:C(G)\to C(F)$ be restriction. Then
\[
 D\pi^*:C(L)\longrightarrow C(G),
 \qquad
 AR:C(G)\longrightarrow C(L)
\]
are positive contractions whose composition on $C(L)$ is the identity.
Composing this retraction with the one given by $J_G,Q_G$ is enough to conclude. 
\end{proof}

% Uncomment and complete before public circulation, if appropriate.
% \subsection*{Acknowledgements}
% The authors thank ...

\section*{AI disclosure}
The authors acknowledge the use of OpenAI's ChatGPT 5.5 in the development of this paper. In particular, ChatGPT provided an initial draft of the proof of the main theorem. The authors subsequently analyzed the draft in detail, considerably simplifying the original arguments, and constructed the paper in its present form. Therefore, all statements, proofs and references appearing in this manuscript have been carefully revised by the authors, who accept the sole responsibility for any remaining inaccuracies or mistakes.

\providecommand{\bysame}{\leavevmode\hbox to3em{\hrulefill}\thinspace}
\providecommand{\MR}{\relax\ifhmode\unskip\space\fi MR }
\providecommand{\MRhref}[2]{%
  \href{http://www.ams.org/mathscinet-getitem?mr=#1}{#2}
}
\providecommand{\href}[2]{#2}

\end{document}